\documentclass[12pt,a4paper]{article}
\usepackage{epsfig,psfrag,amsmath,amssymb,latexsym}
\usepackage{amscd }
\usepackage{amsfonts}
\usepackage{graphicx}
\usepackage{epstopdf}
\usepackage{bbm}
\usepackage{sectsty}
\usepackage{amsthm}

\sectionfont{\large}

\def\P{\mathbb{P}}

\def\E{\mathbb{E}}
\def\R{\mathbb{R}}
\def\N{\mathbb{N}}

\DeclareMathOperator{\sgn}{\mathrm{sgn}}
\DeclareMathOperator{\arcosh}{\mathrm{arcosh}}
\def\h{\mathbf{h}}
\DeclareMathOperator{\Defi}{\mathrel{\mathop:}=}

\newtheorem{theorem}{Theorem}[section]
\newtheorem{lemma}{Lemma}[section]

\theoremstyle{remark}
\newtheorem{remark}{Remark}[section]
\newtheorem{example}{Example}[section]

\numberwithin{equation}{section}

\title{\large MODERATE DEVIATIONS FOR RANDOM FIELD CURIE-WEISS MODELS}
\author{Matthias L\"owe \,and Raphael Meiners\thanks{Research supported by German National Academic Foundation}\;\thanks{Corresponding author: raphael.meiners@uni-muenster.de}
\\
{\small \it Institute for Mathematical Statistics, University of M\"unster, Germany}}

\begin{document}
\maketitle
\vspace{-12pt}
\abstract{\noindent The random field Curie-Weiss model is derived from the classical Curie-Weiss model by replacing the deterministic global magnetic field by random local magnetic fields. This opens up a new and interestingly rich phase structure. In this setting, we derive moderate deviations principles for the random total magnetization $S_n$, which is the partial sum of (dependent) spins. A typical result is that under appropriate assumptions on the distribution of the local external fields there exist a real number $m$, a positive real number $\lambda$, and a positive integer $k$ such that $(S_n-nm)/n^{\alpha}$ satisfies a moderate deviations principle with speed $n^{1-2k(1-\alpha)}$ and rate function $\lambda x^{2k}/(2k)!$, where $1-1/(2(2k-1)) < \alpha < 1$.}
\section{Introduction}\label{Introduction}

Mean-field models such as the classical Curie-Weiss model are important models in statistical mechanics. Even though these models have been introduced as solvable simplifications of nearest-neighbor models they allow an explanation of important physical phenomena such as multiple phases and metastable states. What is more, these models allow to study the behaviour of thermodynamic quantities such as the magnetization close to or at the critical temperature. For the specific case of the classical Curie-Weiss model, Ellis and Newman computed the fluctuations of the magnetization $S_n$ (see \cite{Elli1978a, Elli1978b, Elli1980}). They derived a Law of Large Numbers (LLN for short) and a Central Limit Theorem (CLT) for $S_n$. From a probabilistic point of view, a natural next step after showing a LLN and a CLT was to study large and moderate deviations properties of $S_n$. While a Large Deviations Principle (LDP) already goes back to Ellis (see \cite{Elli1985}), in \cite{Eich2004} Eichelsbacher and L\"owe showed Moderate Deviations Principles (MDPs) for $S_n$.

\noindent Note that technically speaking there is no difference between a LDP and a MDP. However, while a LDP studies the fluctuations on the scale of a LLN, a MDP studies the fluctuations on scales that are between them of a LLN and some -- possibly nonstandard -- CLT. Typical phenomena can be well illustrated on the basis of the MDP for sums of i.\,i.\,d.\ centered random variables $\sum_{i=1}^{n}X_i$ (see \cite{Eich2003}). Under suitable conditions the following holds:
\begin{itemize}
\item $\sum_{i=1}^{n}X_i/n$ satisfies a LDP with rate function $I(x) =  \sup_y\{x y - \ln \E[\exp(y X_1)]\}$.
\item $\sum_{i=1}^{n}X_i/n^{\alpha}, \alpha \in (\frac12,1),$ satisfies a MDP with rate function $I(x) = x^2/(2\E[X_1^2])$.
\item $\sum_{i=1}^{n}X_i/\sqrt{n}$ converges in distribution to a Gaussian with mean 0 and variance $\E[X_1^2]$.
\end{itemize}
Note that the LDP rate function depends on the fine structure of $X_1$ whereas the MDP rate function just depends on $\E[X_1^2]$. This property is inherited from the CLT, where the limiting distribution also just depends on $\E[X_1^2]$. What is more, one observes that the MDP rate function appears as the first term of the Taylor expansion of the logarithmic limiting density in the CLT. However, a MDP also inherits properties from the large deviations such as the exponential decay of probabilities. Although the universality of these properties have never been proved one does find more examples for their existence even if the CLT density is not the normal distribution (see e.\,g.\ \cite{Eich2004, LM01, LMR02, Reic2012}). Note, however that the picture may be different in the setting of disordered models. E.\,g.\ in the Hopfield model there is a non-standard CLT at the critical temperature, while an almost sure moderate deviations principle may fail to hold (see \cite{GeLoe99, EiLoe_hopf}). 

\noindent The purpose of the present paper is to complete the fluctuation picture of $S_n$ in the setting of random field Curie-Weiss models (RFCW). These models are natural extensions of the classical Curie-Weiss model to the effect that the deterministic global external magnetic field is replaced by random local external fields. The models are some of the easiest disordered mean-field models and have been studied intensively over the last decades, see e.\,g.\ \cite{Sali1985, Amar1991, Amar1992} and the references therein. For results on the dynamics or metastates of these models see e.\,g.\ \cite{Kuel1997,Font2000,Bovi2009,Iaco2010}. It may be worthwhile noting that in the present paper, unlike many of the above authors, we will neither assume the random external fields to be symmetrically Bernoulli distributed nor to be bounded at all. The quenched free energy and the annealed free energy of the RFCW have been computed in \cite{Loew2012}. The fluctuations of $S_n$ on the level of a CLT were studied by Amaro de Matos and Perez \cite{Amar1991}. Here, interestingly, the limiting density depends on the distribution of random external field. On the other hand, in \cite{Loew2012} we derived an explicit LDP for $S_n$ (and her as well the rate functions depends on the distribution of the external field). In this paper we prove MDPs for $S_n$ and thus complete the investigation of the fluctuations of $S_n$.

\noindent The outline of the paper is as follows: In Section \ref{Model} we formally introduce random field Curie-Weiss models. In Section \ref{Main results} we state our main results, which are going to be proved in Section \ref{Proofs}. Meanwhile, we apply our results in Section \ref{Examples} and present examples for different choices of the random external field. We consider the cases of the classical Curie-Weiss model with external field $h$ and the Curie-Weiss model with dichotomous external field. For the physical relevance of the latter model see e.\,g.\ \cite{Thom1972}, p. 105. In Section \ref{Auxiliary results} we prepare the proofs of our main results by proving mainly technical auxiliary results.

\section{Random field Curie-Weiss models}\label{Model}

The Curie-Weiss model is a mean-field model of a ferromagnet. Due to its mean-field structure the spatial location of the spins is unimportant. The Hamiltonian of the Curie-Weiss model with external magnetic field $h \in \R$ can therefore be described by
\begin{equation*}
H_{n,h}^{\text{CW}} (\sigma)~=~ -\frac 1 {2n} \sum_{i,j=1}^n  \sigma_i \sigma_j - h \sum_{i=1}^n \sigma_i ~=~ -\frac 1 {2n} S_n^2 -h S_n, \quad \sigma\in\{-1,+1\}^n ,
\end{equation*}
where $S_n=\sum_{i=1}^n \sigma_i $ is the total magnetization of the model. In the RFCW the constant external magnetic field $h$ is now replaced by i.\,i.\,d.\ random variables $(h_i, i \in \N)$
distributed according to $\mu$, the $\N$-fold product measure of the marginal distribution of $h_1$, $\nu$. The Hamiltonian of the RFCW is thus given by
\begin{equation}\label{Hamiltonian}
H_{n,\h}^{\text{RFCW}} (\sigma)~=~ -\frac 1 {2n} \sum_{i,j=1}^n  \sigma_i \sigma_j - \sum_{i=1}^n h_i\sigma_i,  \qquad \sigma\in\{-1,+1\}^n.
\end{equation}
Correspondingly, the RFCW can be associated with the following Gibbs measure on $\{-1,+1\}^n$ at inverse temperature $\beta >0$
\begin{equation*}
P_{n,\beta}^\h(\sigma) ~=~  \frac{1}{Z_{n,\beta}^\h}\exp\left(\frac{\beta}{2n} \left[ \sum_{i=1}^n \sigma_i \right]^2 + \beta\sum_{i=1}^n h_i \sigma_i\right),
\end{equation*}
where
\begin{equation*}
Z_{n,\beta}^\h ~=~ \sum _{\sigma \in \{-1,+1\}^n} \exp\left(\frac{\beta}{2n} \left[ \sum_{i=1}^n \sigma_i \right]^2 + \beta\sum_{i =1}^n h_i \sigma_i\right)
\end{equation*}
is a normalizing constant called partition function.

\noindent In the following, let $\h$ denote the random variable corresponding to the vector $(h_i : i \in \N)$ and $h$ a realization of this vector, say $h = (\tilde{h}_i : i \in \N)$. Sometimes, when it is clear from the context, we also use the notation $h \in \mathbb{R}$ for an integration variable. We write $\rho \sim f(s) ds$ or $\rho \sim Q$ if the measures $\rho$ and $Q$ equal in distribution where $dQ = f(s) ds$. By $\stackrel{w}{\rightarrow}$ we denote weak convergence and by $\mathcal{N}(\tilde{\mu},\tilde{\sigma}^2)$ a normal distribution with mean $\tilde{\mu} \in \R$ and variance $\tilde{\sigma}^2 \in \R_{>0}$. The Borel $\sigma$-algebra of a topological space $X$ is denoted by $\mathcal{B}(X)$ and the power set of a set $S$ by $\mathcal{P}(S)$.

\noindent As mentioned before in the setting of RFCW a lot is already known about the fluctuations of $S_n$. They crucially depend on the function
\begin{equation*}
G(x) ~\Defi~ G_{\beta}^{\nu}(x) ~\Defi~ \frac{\beta}{2} x^2 - \int_{\R}\ln \cosh[\beta(x+h)]d\nu(h),
\end{equation*}
which is real analytic. Therefore, the set of (local or global) minima of $G$ is discrete. This set is also bounded since $G(x) \sim \beta x^2 /2$ as $|x| \to \infty$ and consequently the set is in addition finite and non-empty. We stress this as the minima of $G$ are of particular importance. We shall call $\nu$ pure (and centered at $m$) if $G$ has a unique global minimum (at $m$). What is more, we shall call a real number $m$ a minimum of type $k \in \N$ and strength $\lambda \in \R_{>0}$ if
\begin{equation*}
G(x) ~=~ G(m) +\frac{\lambda}{(2k)!} (x-m)^{2k} +\mathcal{O}\big((x-m)^{2k+1}\big) \quad \text{ as } x\rightarrow m.
\end{equation*}
Note that every minimum of $G$ is of finite type as the real analytic function $G$ would be constant if $G^{(n)}(m) = 0$ for all $n \geq 1$.

\noindent A first result about the fluctuations of $S_n$ has been proved by Amaro de Matos and Perez:

\begin{theorem}[CLT, cf.\ Theorem 2.8 in \cite{Amar1991}]\label{CLT}
Suppose that $\nu$ has a finite second moment, i.\,e.
\begin{equation*}
\int_\R h^2 d\nu(h) ~<~ \infty,
\end{equation*}
and that $\nu$ is not a Dirac measure. Let $\{m_1,\ldots,m_l\}$ be the set of global minima of $G$ and let $k_i$ be the type and $\lambda_i$ be the strength of $m_i,1\leq i\leq l$.
\begin{itemize}
\item[(i)] If $l=1$, then
\begin{equation*}
P_{n,\beta}^h \circ \left(\frac{S_n-nm_1}{n^{1-\frac{1}{2(2k_1-1)}}}\right)^{-1} ~\stackrel{w}{\rightarrow}~
\begin{cases}
\mu_{k_1},& \text{ if } k_1 > 1,\\
\tilde{\mu}_{k_1},& \text{ if } k_1 = 1,
\end{cases}
\end{equation*}
as $n \to \infty$ and for some constants $c_1,c_2>0$,
\begin{eqnarray*}
\mu_{k_1} &\sim& s^{2k_1-2} e^{-c_1s^{2(2k_1-1)}} ds,\\
\tilde{\mu}_{k_1} &\sim& \mathcal{N}(u_1(h),\lambda^{-1}-\beta^{-1}),\\
\mu^{u_1} &\sim& \mathcal{N}(0,c_2).
\end{eqnarray*}
\item[(ii)] If $l>1$, then for every $1\leq i \leq l$ with $k_i = 1$ there exists $A = A(m_i) > 0$ such that for every $0<a<A$
\begin{equation*}
P_{n,\beta}^h \circ \left(\frac{S_n-nm_i}{n^{1-\frac{1}{2(2k_i-1)}}}\bigg|\,\frac{S_n}{n} \in [m_i-a,m_i+a] \right) ~\stackrel{w}{\rightarrow}~ \tilde{\mu}_{k_i},
\end{equation*}
as $n \to \infty$, where for some $c_3>0$
\begin{eqnarray*}
\tilde{\mu}_{k_i} &\sim& \mathcal{N}(u_2(h),\lambda_i^{-1}-\beta^{-1}),\\
\mu^{u_2} &\sim& \mathcal{N}(0,c_3).
\end{eqnarray*}
\end{itemize}
\end{theorem}

\noindent Moreover, we already proved a LDP for $S_n$:

\begin{theorem}[LDP, cf.\ Theorem 3.3 in \cite{Loew2012}]\label{LDP}
Suppose that $\nu$ has a finite absolute first moment, i.\,e.
\begin{equation*}
\int_\R |h| d\nu(h) ~<~ \infty.
\end{equation*}
Then, $\mu$-a.\,s.\ the sequence of measures $(P_{n,\beta}^\h \circ (S_n/n)^{-1})_{n \in \N}$ satisfies a LDP with good rate function $I_\beta^\nu$ defined by
\begin{equation*}
I_\beta^{\nu}(x) ~\Defi~ \sup_{y \in \R} \left\{ G(y) - \frac{\beta}{2}(x-y)^2 \right\} - \inf _{w \in \R} G(w).
\end{equation*}
\end{theorem}

\noindent It is well-known from Large Deviations Theory (see e.\,g. Theorem II.7.2 in \cite{Elli1985}) that the magnetization per spin $S_n/n$ is asymptotically concentrated around the global minima of the rate function $I_\beta^{\nu}$. The global minima of $I_\beta^{\nu}$, which depend on $\beta$ and $\nu$, coincide with the global minima of $G$ (cf. Remark 3.2 in \cite{Loew2012}). Therefore, we identify the following phases of the system:
\begin{itemize}
\item[(i)]{\it Paramagnetic phase}: $G$ has a unique global minimum of type 1.
\item[(ii)]{\it Ferromagnetic phase}: $G$ has two global minima, both of type 1.
\item[(iii)]{\it First-order phase transition}: $G$ has several global minima, all of type 1.
\item[(iv)]{\it Second-order phase transition}: $G$ has a unique global minimum of type 2.
\item[(v)]{\it Tricritical point}: $G$ has a unique global minimum of type 3.
\end{itemize}

\section{Main results}\label{Main results}
Before we state our main results, let us briefly recall the definition of a LDP (see \cite{Demb1998}). We say that a sequence of measures $(\mu_n)_n$ on $\R$ satisfies a LDP with speed $(a_n)_{n \in \N}$ and (good) rate function $I$ if $I$ is a lower semicontinuous function with compact level sets such that for all measureable sets $C \in \mathcal{B}(\R)$
\begin{eqnarray*}
- \inf_{x \in C^\circ} I(x) ~\leq~ \liminf_{n \to \infty} \frac{1}{a_n} \ln \mu_n (C) ~\leq~ \limsup_{n\to \infty} \frac{1}{a_n} \ln \mu_n (C) ~\leq~ - \inf_{x \in \overline{C}} I(x).
\end{eqnarray*}
Note that a rate function is always non-negative as the choice $C = \R$ reveals. Moreover, a sequence of real-valued random variables $(Y_n)_{n \in \N}$ is said to satisfy a LDP with speed $(a_n)_{n \in \N}$ and (good) rate function $I$ if the sequence of their distributions does.

\noindent The purpose of the present paper is to prove a MDP for the magnetization in the setting of random field Curie-Weiss models. Technically, the definition of a MDP is identical to the definition of a LDP. However, it is common to speak of a LDP whenever the scale coincides with the scale of a LLN. In contrast, one speaks of a MDP, if the scale is between the scales of a LLN and some sort of CLT. Considering Theorem \ref{CLT} and Theorem \ref{LDP} we will therefore study the behavior of $(S_n-n m)/n^{\alpha}$ for $1-1/(2(2k-1))<\alpha<1$, where $k$ is the type of the minimum $m$. To state our results we need to introduce two quantities. The height of the minimum $m$ is defined as
\begin{equation*}
h ~\Defi~ h(m) ~\Defi~ G(m) - \inf_{w \in \R}{G(w)} ~(\geq~ 0)
\end{equation*}
and the broadness of $m$ by
\begin{equation*}
b ~\Defi~ b(m) ~\Defi~ \inf_{\substack{y \in \R:\\G(y)<G(m)}}{|y-m|} ~(>~ 0).
\end{equation*}
Our results read as follows:

\begin{theorem}[MDP, conditioned version]\label{conditioned version}
Let $m$ be a (local or global) minimum of $G$ and let $m$ be of type $k$ and strength $\lambda$. Suppose that $\nu$ has a finite second moment and that
\begin{equation}\label{eq:condition}
\beta ~>~ \frac{2 h}{b^2}.
\end{equation}
Then, $\mu$-a.\,s.\ there exists $A = A(m) > 0$ such that for all $0 < a < A$ and every $1-1/\big(2(2k-1)\big)<\alpha<1$ the sequence of measures
\begin{equation*}
\left(P_{n,\beta}^\h \left(\frac{S_n-n m}{n^{\alpha}} \in \bullet \Big|\, \frac{S_n}{n} \in [m-a,m+a]\right)\right)_{n \in \N}
\end{equation*}
satisfies a moderate deviations principle with speed $n^{1-2k(1-\alpha)}$ and rate function
\begin{equation}\label{eq:rate function}
I(x) ~\Defi~ I_{k, \lambda, \beta}(x) ~\Defi~
\begin{cases}
\frac{x^2}{2 \sigma^2},& \text{ if } k =1,\\
\frac{\lambda x^{2k}}{(2 k)!},& \text{ if } k \geq2,\\
\end{cases}
\end{equation}
where $\sigma^2\Defi \lambda^{-1} - \beta^{-1}$.
\end{theorem}

\begin{remark}
If $m$ is a global minimum of $G$, then condition \eqref{eq:condition} holds trivially and there is no condition on the temperature whatsoever. The appearance of condition \eqref{eq:condition} is discussed in more detail in Remark \ref{Remark:Condition}.
\end{remark}

\begin{theorem}[MDP, unconditioned version]\label{unconditioned version}
Suppose that $\nu$ is pure and centered at $m$, where $k$ is the type and $\lambda$ the strength of $m$. Furthermore, suppose that $\nu$ has a finite second moment. Then, $\mu$-a.\,s.\ for every $1-1/\big(2(2k-1)\big)<\alpha<1$ the sequence of measures
\begin{equation*}
\left(P_{n,\beta}^\h \left(\frac{S_n-nm}{n^{\alpha}} \in \bullet \right)\right)_{n \in \N}
\end{equation*}
satisfies a moderate deviations principle with speed $n^{1-2k(1-\alpha)}$ and rate function I given by \eqref{eq:rate function}.
\end{theorem}

\begin{remark}
As we have already mentioned in the introduction, one typically recovers the MDP rate function as the first term of the Taylor expansion of the logarithmic limiting density in the CLT. We find this behavior in the RFCW for minima of type 1. Indeed, for $\mu$-a.\,e.\ $h$ the MDP rate function is given by $x^2/(2 \sigma^2)$ and the limiting density in the CLT is Gaussian with variance $\sigma^2$. However, the RFCW reveals an atypical behavior for minima of larger types in the sense that the mentioned folklore does not hold anymore. If $k \geq 2$ the MDP rate function is proportional to $x^{2k}$ for $\mu$-a.\,e.\ $h$, whereas the limiting density in the CLT is proportional to $x^{2k-2} e^{-c x^{2(2k-1)}}$.
\end{remark}

\begin{remark}\label{remark}
Note that our unconditioned version of a MDP, Theorem \ref{unconditioned version}, is just stated for the case, where $G$ has a unique global minimum. In contrast, Theorem 1.18 in \cite{Eich2004} which describes MDPs for Curie-Weiss models is not restricted to this situation. However, the following Example \ref{counterexample} shows that Theorem 1.18 does not hold in the claimed generality.
\end{remark}

\begin{example}[Low-temperature Curie-Weiss model in the absence of an external field]\label{counterexample}
Take $\nu = \delta_0$ and $\beta > 1$. Let $m > 0$ be the unique positive solution of the fixed-point equation
\begin{equation*}
x ~=~ \tanh (\beta x)
\end{equation*}
Note that $m$ and $-m$ are the global minima of
\begin{equation*}
G(x) ~=~ \frac{\beta}{2}x^2 - \ln \cosh (\beta x),
\end{equation*}
both of type 1 and strength $\lambda = \beta -\beta^2(1-m^2)$. Theorem 1.18 (see also Example 2.1) in \cite{Eich2004} states that for $1/2 < \alpha < 1$
\begin{equation*}
\left(P_{n,\beta}^h \circ \left(\frac{S_n-nm}{n^{\alpha}}\right)^{-1}\right)_{n \in \N} \text{and } \left(P_{n,\beta}^h \circ \left(\frac{S_n+nm}{n^{\alpha}}\right)^{-1}\right)_{n \in \N}
\end{equation*}
satisfy MDPs, both with speed $n^{2\alpha-1}$ and rate function $I(x) = x^2/(2\sigma^2)$, where $\sigma^2 = (1-m^2)/(1-\beta(1-m^2))$. However, this is not true since the probability that $(S_n \boldsymbol{-} nm)/n^{\alpha}$ takes large values is \emph{not} exponentially small on the scale $n^{2\alpha-1}$, what is due to the fact that $(S_n \boldsymbol{+} nm)/n^{\alpha}$ does take small values with some probability. To be precise, the MDP for $S_n-nm/n^{\alpha}$ would imply
\begin{equation}\label{eq:counterexample}
\lim_{n \to \infty} \frac{1}{n^{2\alpha-1}}\ln P_{n,\beta}^h \left(\left|\frac{S_n-nm}{n^{\alpha}}\right| > 1 \right) ~=~ -\frac{1}{2 \sigma^2}.
\end{equation}
But, using the MDP for $S_n+nm/n^{\alpha}$, we see
\begin{eqnarray*}
0
&\geq& \lim_{n \to \infty} \frac{1}{n^{2\alpha-1}}\ln P_{n,\beta}^h \left(\left|\frac{S_n-nm}{n^{\alpha}}\right| > 1 \right) \\
&\geq& \lim_{n \to \infty} \frac{1}{n^{2\alpha-1}}\ln P_{n,\beta}^h \left(\frac{S_n-nm}{n^{\alpha}} < - 1 \right)\\
&\geq& \lim_{n \to \infty} \frac{1}{n^{2\alpha-1}}\ln P_{n,\beta}^h \left(\frac{S_n+nm}{n^{\alpha}} < 1 \right)\\
&=& - \inf_{x < 1} \frac{x^2}{2 \sigma^2}\\
&=& 0,
\end{eqnarray*}
which contradicts \eqref{eq:counterexample}.
\end{example}

\noindent The proofs of our main theorems use ideas that can also be found in \cite{Loew2012, Reic2012, Domb2009}. To put it roughly, the proof is mainly divided into four parts:

\begin{enumerate}
\item Firstly, we transform the measure of interest $P_{n,\beta}^h \circ ((S_n-nm)/n^{\alpha})^{-1}$ by the so-called Hubbard-Stratonovich transformation (see Lemma \ref{transformation}), i.\,e.\ we add a normal-distributed random variable $W$ and consider $P_{n,\beta}^h \circ ((S_n-nm)/n^{\alpha}+W/n^{\alpha-1/2})^{-1}$. The Hubbard-Stratonovich transformation is tailor-made for quadratic interactions such as the ones presented in \eqref{Hamiltonian}.
\item The main advantage of the transformation is that the latter measure is absolutely continuous with respect to Lebesgue measure on $\R$ and that its density is given explicitly (see \eqref{density}). Therefore, the proof of a MDP for $P_{n,\beta}^h \circ ((S_n-nm)/n^{\alpha}+W/n^{\alpha-1/2})^{-1}$ boils down to controlling integrals of exponential functions. As a first step we study these functions in two lemmata (see Lemma \ref{Lemma:deMatos} and Lemma \ref{uniform convergence}).
\item Using ideas of Laplace's method, which is concerned with the asymptotics of integrals of exponential functions, we can then prove a MDP for $P_{n,\beta}^h \circ ((S_n-nm)/n^{\alpha}+W/n^{\alpha-1/2})^{-1}$.
\item Finally, we pull back the result for $P_{n,\beta}^h \circ ((S_n-nm)/n^{\alpha}+W/n^{\alpha-1/2})^{-1}$ to a result for $P_{n,\beta}^h \circ ((S_n-nm)/n^{\alpha})^{-1}$ by means of a so-called {\it transfer principle} (Lemma \ref{transfer principle}).
\end{enumerate}

\section{Examples}\label{Examples}

Before we start with the preparation of the proofs we want to discuss two applications of our results.

\begin{example}[Curie-Weiss model with external field]
Consider the classical Curie-Weiss model with external field $h$, i.\,e.\ take $\nu = \delta_h$ for some $h \neq 0$. Let $m$ be the unique solution of the fixed-point equation
\begin{equation*}
x ~=~ \tanh \big(\beta (x +h )\big)
\end{equation*}
with $\sgn(m) = \sgn(h)$. Then $m$ is the unique minimum of
\begin{equation*}
G(x) ~=~ \frac{\beta}{2} x^2 - \ln \cosh [\beta (x + h)]
\end{equation*}
and it is of type 1 and strength
\begin{equation*}
G^{(2)}(m) ~=~ \beta - \beta^2(1-(\tanh(\beta(m +h )))^2) ~=~ \beta - \beta^2(1-m^2).
\end{equation*}
The model does not show phase transitions, the system is always in the paramagnetic phase. Theorem \ref{unconditioned version} yields that for $1/2 < \alpha < 1$ and all $\beta > 0$ the rescaled magnetization
\begin{equation*}
\frac{S_n-nm}{n^{\alpha}}
\end{equation*}
satisfies a moderate deviations principle with speed $n^{2\alpha-1}$ and rate function
\begin{equation*}
I(x) ~=~ \frac{x^2}{2 \sigma^2},
\end{equation*}
where $\sigma^2 = (1-m^2)/(1-\beta(1-m^2))$. Moreover, the same holds true for the conditional probabilities, i.\,e.\ using Theorem \ref{conditioned version} there exists $A > 0$ such that for all $0 < a < A, 1/2 < \alpha < 1$ and $\beta > 0$
\begin{equation*}
\left(P_{n,\beta}^h \left(\frac{S_n-nm}{n^{\alpha}} \in \bullet \,\bigg|\, \frac{S_n}{n} \in [m-a,m+a]\right)\right)_{n \in \N}
\end{equation*}
satisfies a moderate deviations principle with speed $n^{2\alpha-1}$ and rate function $I(x)$.
\end{example}

\begin{example}[Curie-Weiss model with dichotomous external field]
Consider the Curie-Weiss model with dichotomous external field, i.\,e.\ take $\nu = \frac12 (\delta_h + \delta_{-h})$ for some $h \in \R_{>0}$. In order to study the behaviour of the magnetization on a moderate deviations scale, we need to study the minima of
\begin{equation*}
G(x) ~=~ \frac{\beta}{2} x^2 -\frac12 \ln [\cosh(\beta(x+h))\cosh(\beta(x-h))].
\end{equation*}
This has already been done (see e.\,g.\ chapter 5 in \cite{Amar1992}) and in combination with Theorem \ref{conditioned version} and Theorem \ref{unconditioned version} we get the following:
\begin{itemize}
\item[(i)] If $h\geq 1/2$, then 0 is the only minimum of $G$ and 0 is a minimum of type 1 and strength $\lambda_1 = \beta-\beta^2(1-\tanh(\beta h)^2)$. The system is in the paramagnetic phase. Consequently, $\mu$-a.\,s.\ for $1/2 < \alpha <1$ the rescaled magnetization $$\frac{S_n}{n^{\alpha}}$$ satisfies a MDP with speed $n^{2\alpha-1}$ and rate function $I(x)= x^2/(2\sigma_1^2)$, where $\sigma_1^2 = (1-\tanh(\beta h)^2)/(1-\beta(1-\tanh(\beta h)^2))$. What is more, $\mu$-a.\,s.\ there exists $A > 0$ such that for all $0 < a < A$ and $1/2<\alpha<1$ $$\left(P_{n,\beta}^\h \left(\frac{S_n}{n^{\alpha}} \in \bullet \,\bigg|\, \frac{S_n}{n} \in [-a,a]\right)\right)_{n \in \N}$$ satisfies a MDP with speed $n^{2\alpha-1}$ and rate function $I(x)= x^2/(2\sigma_1^2)$.
\end{itemize}
If $h < 1/2$, then the situation is a bit more subtle. One finds a strictly increasing function $f: [0,1/2) \rightarrow \R$ with $f(0) = 1$ and $f(x_n) \rightarrow \infty$ as $x_n \nearrow 1/2$ such that:
\begin{itemize}
\item[(ii)] If $\beta < f(h)$, then the same result as in (i) holds.
\item[(iii)] If $\beta > f(h)$, then $G$ has two symmetric global minima $m$ and $-m$, where $m$ is the positive solution of the fixed-point equation
\begin{equation}\label{eq}
2 \, m ~=~ \tanh\big(\beta(m+h)\big) + \tanh\big(\beta(m-h)\big).
\end{equation}
The minima are of type 1 and strength $\lambda_2 = \beta-2m\beta^2(\tanh(2\beta m)^{-1}-m)$. The system is in the ferromagnetic phase. Therefore, $\mu$-a.\,s.\ there exists $A > 0$ such that for all $0 < a < A$ and $1/2< \alpha<1$
\begin{eqnarray*}
&\left(P_{n,\beta}^\h \left(\frac{S_n-nm}{n^{\alpha}} \in \bullet \,\bigg|\, \frac{S_n}{n} \in [m-a,m+a]\right)\right)_{n \in \N}& \text{ resp. }\\
&\left(P_{n,\beta}^\h \left(\frac{S_n+nm}{n^{\alpha}} \in \bullet \,\bigg|\, \frac{S_n}{n} \in [-m-a,-m+a]\right)\right)_{n \in \N}&
\end{eqnarray*}
satisfy MDPs, both with speed $n^{2\alpha-1}$ and rate function $I(x)= x^2/(2\sigma_2^2)$, where $\sigma_2^2 =2m (\tanh(2\beta m)^{-1}-1)/(2m\beta(\tanh(2\beta m)^{-1}-m)-1)$.
\item[(iv)] On the critical line $\beta = f(h)$ the situation is even more delicate. If $h < h_c \Defi \frac23 \arcosh \sqrt{3/2}$ we note a second-order phase transition. 0 is the only minimum of $G$, it is of type 2 and strength $\lambda_3 = 2 \beta^4(1-4\tanh(\beta h)^2+3\tanh(\beta h)^4)$. Thus, $\mu$-a.\,s\ there exists $A > 0$ such that for all $0 < a < A$ and $5/6 < \alpha <1$ $$\left(P_{n,\beta}^\h \left(\frac{S_n}{n^{\alpha}} \in \bullet \,\bigg|\, \frac{S_n}{n} \in [-a,a]\right)\right)_{n \in \N}$$ satisfies a MDP with speed $n^{4\alpha-3}$ and rate function $I(x)= \lambda_3 x^4/12$. At the tricritical point $h = h_c$ 0 is again the only minimum of $G$, but of type 3 and strength $\lambda_4 = 8 \beta^6(-2+17\tanh(\beta h)^2-30 \tanh(\beta h)^4+15 \tanh(\beta h)^8)$. Consequently, $\mu$-a.\,s.\ there exists $A > 0$ such that for all $0 < a < A$ and $9/10 < \alpha <1$ $$\left(P_{n,\beta}^\h \left(\frac{S_n}{n^{\alpha}} \in \bullet \,\bigg|\, \frac{S_n}{n} \in [-a,a]\right)\right)_{n \in \N}$$ satisfies a MDP with speed $n^{6\alpha-5}$ and rate function $I(x)= \lambda_4 x^6/720$. Finally, if $h > h_c$ we note a first-order phase transition. $G$ has three global minima, all of type 1. 0 is a minimum of strength $\lambda_1$ and $m$ resp $-m$, where $m$ is the positive solution of \eqref{eq}, are minima of strength $\lambda_2$. Therefore, $\mu$-a.\,s.\ there exists $A > 0$ such that for all $0 < a < A$ and $1/2 < \alpha <1$ $$\left(P_{n,\beta}^\h \left(\frac{S_n}{n^{\alpha}} \in \bullet \,\bigg|\, \frac{S_n}{n} \in [-a,a]\right)\right)_{n \in \N}$$ satisfies a MDP with speed $n^{2\alpha-1}$ and rate function $I(x)= x^2/(2\sigma_1^2)$, whereas
\begin{eqnarray*}
&\left(P_{n,\beta}^\h \left(\frac{S_n-nm}{n^{\alpha}} \in \bullet \,\bigg|\, \frac{S_n}{n} \in [m-a,m+a]\right)\right)_{n \in \N}& \text{ resp. }\\
&\left(P_{n,\beta}^\h \left(\frac{S_n+nm}{n^{\alpha}} \in \bullet \,\bigg|\, \frac{S_n}{n} \in [-m-a,-m+a]\right)\right)_{n \in \N}&
\end{eqnarray*}
satisfy MDPs with speed $n^{2\alpha-1}$ and rate function $I(x)= x^2/(2\sigma_2^2)$.
\end{itemize}

\noindent The same analysis could have even been done for a possibly unsymmetrical dichotomous external field by using results of \cite{Kuel2007}. Considering $\nu = t \delta_h +(1-t) \delta_{-h}$ for some $t \in [0,1]$ would have led to the study of
\begin{equation*}
G(x) ~=~ \frac{\beta}{2} x^2 - t \ln\big(\cosh [\beta (x + h)] \big) -(1-t) \ln \big(\cosh [\beta (x - h)]\big),
\end{equation*}
which can be found in Theorem 4.1 of \cite{Kuel2007}.

\end{example}

\section{Auxiliary results}\label{Auxiliary results}
In order to make the ideas of the proofs of our main results easier accessible, we state several, mainly technical, lemmas in this section. As we already mentioned, the first step of our proof is always to perform a transformation to gain better control of the measures of interest. This transformation is sometimes called Hubbard-Stratonovich transformation and it works the following way:

\begin{lemma}[cf.\ Lemma 2.3 in \cite{Amar1991}]\label{transformation}
Let $m \in \R$ and $\alpha \in (0,1)$ be real numbers. Then, for every realization $h = (\tilde{h}_i : i \in \N)$ and every $n \geq 1$, the random variable
\begin{equation*}
\frac{S_n -n m}{n^{\alpha}} + \frac{W}{n^{\alpha-\frac12}}
\end{equation*}
is, under the measure $P_{n,\beta}^h$, absolutely continuous with Lebesgue density given by
\begin{equation}\label{density}
\frac{e^{-nG_n^h(m+ n^{\alpha-1} \bullet)}}{\int_\R e^{-nG_n^h(m+n^{\alpha-1} s)}ds},
\end{equation}
where $G_n^h(x)$ is defined by
\begin{equation*}
G_n^h(x) ~\Defi~ \frac{\beta}{2}x^2 - \frac{1}{n} \sum_{i=1}^n \ln \cosh [\beta(x + \tilde{h}_i)]
\end{equation*}
and $W$ is a Gaussian random variable with mean zero and variance $\beta^{-1}$, which is defined on (an extension of) $(\{\pm1\}^n, \mathcal{P}(\{\pm1\}^n),P_{n,\beta}^h)$ and which is independent of the sequence $(S_n)_{n \in \N}$.
\end{lemma}

\noindent As we are not interested in the moderate deviations behavior of $(S_n -n m)/n^{\alpha} + W/n^{\alpha-1/2}$ we need to find a way how to compensate for the convolution. The following {\it transfer principle} will tell us how to get a MDP for $(S_n -n m)/n^{\alpha}$ from a MDP for its Hubbard-Stratonovich transform:

\begin{lemma}\label{transfer principle}
Let $m$ be a (local or global) minimum of $G$ and let $m$ be of type $k$ and strength $\lambda$.
\begin{itemize}
\item[(i)] Suppose that $\mu$-a.\,s. $$\left(P_{n,\beta}^\h \circ \left(\frac{S_n-nm}{n^{\alpha}} + \frac{W}{n^{\alpha-\frac12}}\right)^{-1}\right)_{n \in \N}$$
satisfies a moderate deviations principle with speed $n^{1-2k(1-\alpha)}$ and rate function \begin{equation}\label{eq:rate function J}
J(x) ~\Defi~ J_{\lambda, k}(x) ~\Defi~ \frac{\lambda x^{2k}}{(2k)!}.
\end{equation}
Then, $\mu$-a.\,s. $$\left(P_{n,\beta}^\h \circ \left(\frac{S_n-nm}{n^{\alpha}}\right)^{-1}\right)_{n \in \N}$$ satisfies a moderate deviations principle with speed $n^{1-2k(1-\alpha)}$ and rate function $I$ given by \eqref{eq:rate function}.
\item[(ii)] Suppose that \eqref{eq:condition}\ holds, i.\,e.
\begin{equation*}
\beta ~>~ \frac{2 h}{b^2}.
\end{equation*}
Let $c$ be the supremum of all $x \in (0, (b  -\sqrt{2 h/\beta})/2]$ such that $m$ is the only minimum of $G$ in $[m-x,m+x]$ and fix $0<a<c$. Suppose that $\mu$-a.\,s.
\begin{equation*}
\left(P_{n,\beta}^\h \left(\frac{S_n-n m}{n^{\alpha}} + \frac{W}{n^{\alpha-\frac12}} \in \bullet \Big|\, \frac{S_n-n m}{n^{\alpha}} + \frac{W}{n^{\alpha-\frac12}} \in [-a n^{1-\alpha},a n^{1-\alpha}]\right)\right)_{n \in \N}
\end{equation*}
satisfies a moderate deviations principle with speed $n^{1-2k(1-\alpha)}$ and rate function $J$ given by \eqref{eq:rate function J}. Then, $\mu$-a.\,s.
\begin{equation*}
\left(P_{n,\beta}^\h \left(\frac{S_n-n m}{n^{\alpha}} \in \bullet \Big|\, \frac{S_n}{n} \in [m-a,m+a]\right)\right)_{n \in \N}
\end{equation*}
satisfies a moderate deviations principle with speed $n^{1-2k(1-\alpha)}$ and rate function $I$ given by \eqref{eq:rate function}.
\end{itemize}
\end{lemma}

\begin{proof}
ad (i): Suppose that $h$ is such that $$\left(P_{n,\beta}^h \circ \left(\frac{S_n-nm}{n^{\alpha}} + \frac{W}{n^{\alpha-\frac12}}\right)^{-1}\right)_{n \in \N}$$ satisfies a moderate deviations principle with speed $n^{1-2k(1-\alpha)}$ and rate function $J$. Let us first consider the case $k=1$. Using Proposition A.1 from \cite{Eich2004} we see that $(P_{n,\beta}^h \circ ((S_n-nm)/n^{\alpha})^{-1})_{n \in \N}$ satisfies a MDP with speed $n^{2\alpha-1}$ and rate function given by
\begin{equation*}
y ~\mapsto~ \sup_{x \in \R}\left(\frac{\lambda x^2}{2}-\frac{\beta (x-y)^2}{2}\right) ~=~ \frac{y^2}{2 \sigma^2} ~=~ I(y),
\end{equation*}
where $\sigma^2\Defi \lambda^{-1} - \beta^{-1}$. If $k \geq 2$ we have $1-2k(1-\alpha)<2\alpha-1$ and the influence of the Gaussian random variable vanishes for $n \to \infty$. Using Lemma 3.4 from \cite{Eich2004} we see that $(P_{n,\beta}^h \circ ((S_n-nm)/n^{\alpha})^{-1})_{n \in \N}$  satisfies a MDP with speed $n^{1-2k(1-\alpha)}$ and rate function $J(y) = I(y)$.\\
\noindent ad (ii): Let $X_n \Defi (S_n-nm)/n^{\alpha}, Y_n \Defi W/n^{\alpha-1/2}$ and $B_n \Defi [-a n^{1-\alpha},a n^{1-\alpha}].$ Choose $h$ such that
\begin{equation*}
\left(P_{n,\beta}^h \left(\frac{S_n}{n} + \frac{W}{\sqrt{n}} \in \bullet \right)\right)_{n \in \N} \text{ resp.} \left(P_{n,\beta}^h \left(\frac{S_n}{n} \in \bullet \right)\right)_{n \in \N}
\end{equation*}
satisfy LDPs with speed $n$ and rate function $I_{\beta}^{\nu}$ resp. $G(x)-\inf_{w \in \R} G(w)$. This can be done with probability 1 due to Theorem 3.3 in \cite{Loew2012} and its proof. Moreover, choose $h$ such that
\begin{equation*}
\left(P_{n,\beta}^h (X_n + Y_n \in \bullet \Big|\, X_n +Y_n \in B_n)\right)_{n \in \N}
\end{equation*}
satisfies a moderate deviations principle with speed $n^{1-2k(1-\alpha)}$ and rate function $J$, which can be done with probability 1 by assumption. We start by showing that the same MDP holds for
\begin{equation*}
\left(P_{n,\beta}^h (X_n + Y_n \in \bullet \Big|\, X_n \in B_n)\right)_{n \in \N}.
\end{equation*}
In order to do this, we show that the two sequences are exponentially equivalent on the scale $n^{1-2k(1-\alpha)}$, i.\,e.
\begin{equation}\label{eq:exp}
\limsup_{n \to \infty} \frac{1}{n^{1-2k(1-\alpha)}}\log \rho_n ~=~ -\infty,
\end{equation}
where
\begin{equation*}
\rho_n ~\Defi~ \sup_{B \in \mathcal{B}(\R)}\big\{P_{n,\beta}^h (X_n +Y_n \in B |\, X_n \in B_n)-P_{n,\beta}^h (X_n + Y_n \in B |\, X_n + Y_n\in B_n)\big\}.
\end{equation*}
Note that for every $B \in \mathcal{B}(\R)$
\begin{eqnarray*}
&& P_{n,\beta}^h (X_n + Y_n \in B |\, X_n \in B_n)-P_{n,\beta}^h (X_n + Y_n \in B |\, X_n +Y_n \in B_n) \\
&\leq& P_{n,\beta}^h (|Y_n| > n^{(1-\alpha)/2}) + P_{n,\beta}^h (X_n + Y_n \in B,|Y_n| \leq n^{(1-\alpha)/2} |\, X_n \in B_n)\\
&& -P_{n,\beta}^h (X_n + Y_n \in B |\, X_n +Y_n \in B_n)\\
&=& P_{n,\beta}^h (|Y_n| > n^{(1-\alpha)/2})+ \left(\frac{1}{P_{n,\beta}^h (X_n \in B_n)}-\frac{1}{P_{n,\beta}^h (X_n +Y_n \in B_n)}\right) \\
&& \times \ P_{n,\beta}^h (X_n + Y_n \in B, X_n \in B_n, |Y_n| \leq n^{(1-\alpha)/2})\\
&& + \frac{P_{n,\beta}^h (X_n + Y_n \in B, X_n \in B_n, |Y_n| \leq n^{(1-\alpha)/2}) - P_{n,\beta}^h (X_n + Y_n \in B \cap B_n)}{P_{n,\beta}^h (X_n +Y_n \in B_n)}
\end{eqnarray*}
and consequently $\rho_n$ is bounded by
\begin{eqnarray*}
&& P_{n,\beta}^h (|Y_n| > n^{(1-\alpha)/2}) + \frac{P_{n,\beta}^h (X_n +Y_n \in B_n)-P_{n,\beta}^h (X_n\in B_n)}{P_{n,\beta}^h (X_n +Y_n \in B_n)}\vee0\\
&& + \frac{P_{n,\beta}^h (X_n + Y_n \in[-an^{1-\alpha}-n^{(1-\alpha)/2},-an^{1-\alpha}]\cup[an^{1-\alpha},an^{1-\alpha}+n^{(1-\alpha)/2}])}{P_{n,\beta}^h (X_n +Y_n \in B_n)}.
\end{eqnarray*}
Using Lemma 1.2.15 from \cite{Demb1998} \eqref{eq:exp} follows from proving that every of the three summands converges to $-\infty$ on a logarithmic scale of order $n^{1-2k(1-\alpha)}$.

\noindent First,
\begin{equation*}
\limsup_{n \to \infty} \frac{1}{n^{1-2k(1-\alpha)}}\log P_{n,\beta}^h (|Y_n| > n^{(1-\alpha)/2}) ~=~ -\infty
\end{equation*}
follows immediately from the standard estimate
\begin{equation*}
\P( Z > x) ~\leq~ \frac{1}{\sqrt{2\pi}x}e^{-\frac12 x^2}
\end{equation*}
for a standard Gaussian $Z, x>0$.

\noindent Second, with $\delta = b - c$ it is
\begin{eqnarray*}
&& \frac{P_{n,\beta}^h (X_n +Y_n \in B_n)-P_{n,\beta}^h (X_n\in B_n)}{P_{n,\beta}^h (X_n +Y_n \in B_n)}\vee0\\
&=& \frac{P_{n,\beta}^h (S_n/n +W/\sqrt{n} \in [m-a,m+a])-P_{n,\beta}^h (S_n/n\in [m-a,m+a])}{P_{n,\beta}^h (S_n/n +W/\sqrt{n} \in [m-a,m+a])}\vee0\\
&\leq& \frac{P_{n,\beta}^h (S_n/n \in [m-a-\delta,m-a]\cup[m+a,m+a+\delta])}{P_{n,\beta}^h (S_n/n +W/\sqrt{n} \in [m-a,m+a])}\\
&& + \frac{P_{n,\beta}^h (|W/\sqrt{n}|>\delta)}{P_{n,\beta}^h (S_n/n +W/\sqrt{n} \in [m-a,m+a])}.
\end{eqnarray*}
Using Lemma 1.2.15 in \cite{Demb1998} we can again consider the two terms separately. We find
\begin{eqnarray*}
&&\lim_{n \to \infty} \frac{1}{n}\log \frac{P_{n,\beta}^h (S_n/n \in [m-a-\delta,m-a]\cup[m+a,m+a+\delta])}{P_{n,\beta}^h (S_n/n +W/\sqrt{n} \in [m-a,m+a])}\\
&=& -\inf_{x \in [m-a-\delta,m-a]\cup[m+a,m+a+\delta]}I_{\beta}^{\nu}(x)+\inf_{x \in [m-a,m+a]}G(x) - \inf_{x\in\R}G(x)\\
&\leq& -\inf_{x \in [m-a-\delta,m-a]\cup[m+a,m+a+\delta]}G(x)+G(m)\\
&<& 0
\end{eqnarray*}
and
\begin{eqnarray*}
&&\lim_{n \to \infty} \frac{1}{n}\log \frac{P_{n,\beta}^h (|W/\sqrt{n}|>\delta)}{P_{n,\beta}^h (S_n/n +W/\sqrt{n} \in [m-a,m+a])}\\
&=& -\frac{\beta \delta^2}{2} +\inf_{x \in [m-a,m+a]}G(x) - \inf_{x\in\R}G(x)\\
&=& -\frac{\beta \delta^2}{2} + h\\
&<& 0.
\end{eqnarray*}
Consequently,
\begin{equation*}
\limsup_{n \to \infty} \frac{1}{n^{1-2k(1-\alpha)}}\log \frac{P_{n,\beta}^h (X_n +Y_n \in B_n)-P_{n,\beta}^h (X_n\in B_n)}{P_{n,\beta}^h (X_n +Y_n \in B_n)}\vee0 ~=~ -\infty.
\end{equation*}

\noindent Finally, since $m$ is the only minimum of $G$ in $[m-a,m+a]$ we can choose $\tilde{a} > a$ such that $m$ is also the only minimum of $G$ in $[m-\tilde{a},m+\tilde{a}]$. Note that
\begin{eqnarray*}
&& \limsup_{n \to \infty} \frac{1}{n}\log P_{n,\beta}^h (X_n + Y_n \in[-an^{1-\alpha}-n^{(1-\alpha)/2},-an^{1-\alpha}]\\
&& \cup[an^{1-\alpha},an^{1-\alpha}+n^{(1-\alpha)/2}])\\
&\leq& \limsup_{n \to \infty} \frac{1}{n}\log P_{n,\beta}^h (S_n/n +W/\sqrt{n} \in [m-\tilde{a},m-a]\cup[m+a,m+\tilde{a}])\\
&=& -\inf_{x \in [m-\tilde{a},m-a]\cup[m+a,m+\tilde{a}]}G(x)+\inf_{x\in\R}G(x)\\
&<& -G(m)+\inf_{x\in\R}G(x)\\
&=& -\inf_{x \in [m-a,m+a]}G(x) + \inf_{x\in\R}G(x)\\
&=& \lim_{n \to \infty} \frac{1}{n} \log P_{n,\beta}^h (S_n/n + W/\sqrt{n} \in [m-a,m+a])\\
&=& \lim_{n \to \infty} \frac{1}{n} \log P_{n,\beta}^h (X_n + Y_n \in B_n)
\end{eqnarray*}
and therefore \eqref{eq:exp} follows.

\noindent Now that we know that
\begin{equation*}
\left(P_{n,\beta}^h (X_n + Y_n \in \bullet \Big|\, X_n \in B_n)\right)_{n \in \N}
\end{equation*}
satisfies a moderate deviations principle with speed $n^{1-2k(1-\alpha)}$ and rate function $J$, we can easily deduce the assertion of (ii) by using slight generalizations of Proposition A.1 ($k =1$) resp. Lemma 3.4 ($k \geq 2$) in \cite{Eich2004}.
\end{proof}

\begin{remark}\label{Remark:Condition}
As it is displayed in Lemma \ref{transfer principle} (ii), the conditioned transfer principle in the setting of a global minimum $m$ holds without any further conditions. Having said that, it seems at least surprising that a condition like \eqref{eq:condition} needs no be imposed to prove a conditioned transfer principle for local minima. However, a closer look reveals that such a condition is natural. As it is shown in the proof, conditioned on the same event $\{S_n/n \in [m-a,m+a]\} \ S_n/n^{\alpha}+W/n^{\alpha-1/2}$ and $S_n/n^{\alpha}$ are closely related since their difference is given by a centered Gaussian. Therefore, the transfer principle comes back to relating
\begin{equation}\label{eq:trans}
(S_n/n^{\alpha}+W/n^{\alpha-1/2}) \big|_{\frac{S_n}{n} \in [m-a,m+a]} \text{ and } (S_n/n^{\alpha}+W/n^{\alpha-1/2})\big|_{\frac{S_n}{n} + \frac{W}{\sqrt{n}} \in [m-a,m+a]}.
\end{equation}
Whereas the event $\{S_n/n \in [m-a,m+a]\}$ does not contain information about $W/n^{\alpha-1/2}$, we see that $W/n^{\alpha-1/2} \approx b n^{1-\alpha} \rightarrow \infty$ on the set $\{S_n/n + W/\sqrt{n} \in [m-a,m+a]\}$ for large values of $h$. Indeed, for large values of $h$ the LDP for $S_n/n$ yields that $S_n/n$ is concentrated around a global minimum of $I_{\beta}^{\nu}$, which in turn is a global minimum of $G$. If $S_n/n + W/\sqrt{n}$ is now concentrated around a local minimum $m$ of $G$, this implies that $W/\sqrt{n}$ is of constant order. In conclusion, for large values of $h$ the different conditioning events lead to different behaviors of $W/n^{\alpha-1/2}$, which in turn lead to different behaviors of the random variables in \eqref{eq:trans} on moderate deviations scales.

\noindent Even more, condition \eqref{eq:condition} seems to be optimal. Since $W/\sqrt{n} \sim \mathcal{N}(0,(\beta n)^{-1})$ a use of the LDP for $S_n/n$ yields
\begin{equation*}
P_{n,\beta}^h \left(\frac{S_n}{n} +\frac{W}{\sqrt{n}} \in [m-a,m+a]\right) ~\approx~ \int_{\R} e^{-n f(\delta)}d\delta,
\end{equation*}
where $f(\delta) = \beta \delta^2 /2 + I_{\beta}^{\nu}(m-\delta)$. The structure of $f$ shows the competing influence of $S_n/n$ and $W/\sqrt{n}$. It is $f(0) = I_{\beta}^{\nu}(m) = G(m)-\inf_{x \in \R}G(x) = h$ and $f(\pm b) = \beta b^2/2$. Thus, if condition \eqref{eq:condition} is satisfied $f$ is minimal around $0$ and major contributions to the event $\{S_n/n +W/\sqrt{n} \in [m-a,m+a]\}$ stem from the event $\{W/\sqrt{n} \approx 0\}$. On the other hand, if condition \eqref{eq:condition} is not satisfied $f$ is minimal for $\pm b$ and consequently $|W/\sqrt{n}| \approx b$.
\end{remark}

\noindent By means of the last lemma the proof of a MDP for $(S_n -n m)/n^{\alpha}$ comes back to the proof of a MDP for $(S_n -n m)/n^{\alpha} + W/n^{\alpha-1/2}$. As we already know, the density of this random variable is proportional to $\exp(-nG_n^h(m+ n^{\alpha-1}s))ds$. Therefore, we study (the asymptotic behavior of) $G_n^h$ in the following two lemmas:

\begin{lemma}\label{Lemma:deMatos}
Suppose that $\nu$ has a finite absolute first moment. Then, for $\mu$-a.\,e.\ realization $h$ the following holds:
\begin{itemize}
\item[(i)] For every $j \in \N_0$
\begin{equation}\label{eq:uniform convergence on compacta}
G_n^{h (j)} ~\rightarrow~ G^{h (j)}
\end{equation}
as $n \to \infty$, where the convergence is uniformly on compact sets of $\R$.
\item[(ii)] Let $m$ be a global minimum of $G$ and let $V$ be a closed (possibly unbounded) subset of $\R$ containing no global minimum of $G$. Then, there exists $\varepsilon > 0$ such that
\begin{equation*}
\int_{V}{e^{-n(G_n^h(s)-G(m))}ds} ~\leq~ e^{-n \varepsilon},
\end{equation*}
where $n$ is sufficiently large.
\end{itemize}
\end{lemma}

\begin{proof}
This is already known due to \cite{Amar1991}. $(i)$ is the assertion of Lemma 3.5, $(ii)$ follows from Lemma 3.1.
\end{proof}

\begin{lemma}\label{uniform convergence}
Let $m$ be a (local or global) minimum of $G$ and let $m$ be of type $k$ and strength $\lambda$. Suppose that $\nu$ has a finite second moment. Then, for $\mu$-a.\,e.\ realization $h$ and every $1-1/\big(2(2k-1)\big)<\alpha<1$ the following holds:
\begin{itemize}
\item[(i)] For every $s \in \R$
\begin{equation*}
n^{2k(1-\alpha)} \big(G_n^h(m+ s n^{\alpha-1})-G_n^h(m)\big) ~\rightarrow~ \frac{\lambda s^{2k}}{(2k)!}
\end{equation*}
as $n \to \infty$. What is more, this convergence holds uniformly on compact sets of $\R$.
\item[(ii)] There exists $\delta > 0$ such that for $n$ sufficiently large
\begin{equation}\label{eq:(ii)}
n^{2k(1-\alpha)} \big(G_n^h(m+ s n^{\alpha-1})-G_n^h(m)\big) ~\geq~ P_{2k}(s)
\end{equation}
for all $s \in [-\delta n^{1-\alpha}, \delta n^{1-\alpha}]$, where $P_{2k}(s) \Defi \frac{\lambda}{2\, (2k)!}s^{2k} -\sum_{i=1}^{2k-1}|s|^{i}$ is a polynomial of degree $2k$.
\end{itemize}
\end{lemma}

\begin{proof}
Let $h$ be such that for all $1 \leq i \leq 2k$ and $1-1/\big(2(2k-1)\big)<\alpha<1$
\begin{eqnarray}\label{convergence of moments}
&& n^{(2k-i)(1-\alpha)} |G_n^{h (i)}(m)-G^{(i)}(m)| \nonumber\\
&=&  \Big|\frac{\sum_{i=1}^n \ln \cosh [\beta(m + \tilde{h}_i)]-\E_{\mu}[\ln \cosh [\beta(m + h_i)]]}{n^{1-(2k-i)(1-\alpha)}}\Big| \nonumber\\
&\rightarrow& 0
\end{eqnarray}
as $n \to \infty$. Since $\nu$ has a finite second moment and $1-(2k-i)(1-\alpha) > 1/2$, we can choose such a realization $h$ with probability 1 due to the strong law of large numbers by Marcinkiewicz and Zygmund (see e.\,g.\ Theorem 2, p.\ 122, in \cite{Chow2003}). Moreover, let $h$ be such that \eqref{eq:uniform convergence on compacta} is satisfied and fix $1-1/\big(2(2k-1)\big)<\alpha<1$.\\
\noindent ad (i): A use of Taylor's theorem gives
\begin{equation}\label{Taylor}
n^{2k(1-\alpha)} \big(G_n^h(m+ s n^{\alpha-1})-G_n^h(m)\big) ~=~ n^{2k(1-\alpha)} \sum_{i=1}^{2k}\frac{G_n^{h (i)}(m)}{i!}\left(s n^{\alpha-1}\right)^{i} +R_n(s),
\end{equation}
where
\begin{eqnarray}\label{eq:Restglied Taylor}
R_n(s)
&=& n^{2k(1-\alpha)} \frac{G_n^{h (2k+1)}(\xi_n)}{(2k+1)!}(s n^{\alpha-1})^{2k+1}\nonumber\\
&=& \frac{1}{n^{1-\alpha}} \frac{G_n^{h (2k+1)}(\xi_n)}{(2k+1)!} s^{2k+1}
\end{eqnarray}
for some $\xi_n \in [m,m+sn^{\alpha-1}]$ if $s\geq 0$ and $\xi_n \in [m+sn^{\alpha-1},m]$ otherwise. Due to \eqref{eq:uniform convergence on compacta}, $G_n^{h (2k+1)}$ converges uniformly on compact sets to $G^{(2k+1)}$, which is a continuous function. Consequently, $R_n$ converges to 0 uniformly on compact sets. What is left to prove is that
\begin{equation}\label{eq:rest}
\sum_{i=1}^{2k}\frac{n^{(2k-i)(1-\alpha)}G_n^{h (i)}(m)}{i!}s^i ~\rightarrow~ \frac{\lambda s^{2k}}{(2k)!}
\end{equation}
uniformly on compact sets as $n \to \infty$. Since $m$ is a minimum of type $k$ and strength $\lambda$ of $G$, i.\,e.
\begin{equation*}
G^{(i)}(m) ~=~
\begin{cases}
0,& \text{ if } i \in \{1,\ldots,2k-1\},\\
\lambda,& \text{ if } i = 2k,
\end{cases}
\end{equation*}
a use of \eqref{convergence of moments} yields
\begin{equation}\label{eq:convergence of derivatives}
\lim_{n \to \infty} n^{(2k-i)(1-\alpha)}G_n^{h (i)}(m) ~=~
\begin{cases}
0,& \text{ if } i \in \{1,\ldots,2k-1\},\\
\lambda,& \text{ if } i = 2k,
\end{cases}
\end{equation}
and \eqref{eq:rest} follows.\\
\noindent ad (ii): Using the Taylor expansion \eqref{Taylor} we see
\begin{eqnarray}\label{eq:1}
&& n^{2k(1-\alpha)} \big(G_n^h(m+ s n^{\alpha-1})-G_n^h(m)\big) \nonumber\\
&=& n^{2k(1-\alpha)} \sum_{i=1}^{2k}\frac{G_n^{h (i)}(m)}{i!}\left(s n^{\alpha-1}\right)^{i} +R_n(s)\nonumber\\
&\geq& \frac{G_n^{h (2k)}(m)}{(2k)!}s^{2k} - \sum_{i=1}^{2k-1}\frac{|n^{(2k-i)(\alpha-1)}G_n^{h (i)}(m)|}{i!}|s|^{i} -|R_n(s)| \nonumber\\
&\geq& \frac{3 \lambda}{4 (2k)!}s^{2k} - \sum_{i=1}^{2k-1}|s|^{i} -|R_n(s)|
\end{eqnarray}
for $n$ sufficiently large, where we have used \eqref{eq:convergence of derivatives} to derive the last line. Finally, we see that \eqref{eq:Restglied Taylor} yields for all $s \in [- \delta n^{1-\alpha},\delta n^{1-\alpha}], \delta \leq 1,$
\begin{equation*}
|R_n(s)| ~\leq~ \delta \, \max_{|x-m|\leq1}|G^{(2k+1)}(x)| \frac{s^{2k}}{(2k)!}
\end{equation*}
for $n$ sufficiently large. Therefore, choosing $\delta$ less than $\lambda \,\max_{|x-m|\leq1}|G^{(2k+1)}(x)| /4 \wedge 1$ we get the assertion by means of \eqref{eq:1}.
\end{proof}
\section{Proofs}\label{Proofs}

This section is devoted to the proofs of our main results, Theorem \ref{conditioned version} and Theorem \ref{unconditioned version}. We will first prove Theorem \ref{conditioned version}. Using parts of this proof, the proof of Theorem \ref{unconditioned version} won't be difficult in the end.

\begin{proof}[Proof of Theorem \ref{conditioned version}]
Let $m$ be a (local or global) minimum of $G$ and let $m$ be of type $k$ and strength $\lambda$. Moreover, let $h$ be such that the assertions of Lemma \ref{uniform convergence} hold and take $A$ to be the minimum of $\delta$ (cf.\ Lemma \ref{uniform convergence}) and c (cf.\ Lemma \ref{transfer principle} (ii)). Fix $1-1/\big(2(2k-1)\big)<\alpha<1$ and $0 < a < A$. Note that by Lemma \ref{uniform convergence}
\begin{equation}\label{eq:erste Abschaetzung}
n^{2k(1-\alpha)} \big(G_n^h(m+ s n^{\alpha-1})-G_n^h(m)\big) ~\geq~ P_{2k}(s)
\end{equation}
for all $|s| \leq a n^{1-\alpha}$ and $n$ sufficiently large.

\noindent In order to prove Theorem \ref{conditioned version}, it suffices by means of Lemma \ref{transfer principle} to prove that
\begin{equation*}
\left(P_{n,\beta}^h \left(\frac{S_n-n m}{n^{\alpha}} + \frac{W}{n^{\alpha-\frac12}} \in \bullet \Big|\, \frac{S_n-n m}{n^{\alpha}} + \frac{W}{n^{\alpha-\frac12}} \in [a n^{1-\alpha},a n^{1-\alpha}]\right)\right)_{n \in \N}
\end{equation*}
satisfies a moderate deviations principle with speed $n^{1-2k(1-\alpha)}$ and rate function
\begin{equation*}
J(x) ~=~\frac{\lambda x^{2k}}{(2k)!}.
\end{equation*}
Note that the densities of these measures are given in Lemma \ref{transformation} and thus it suffices to prove that
\begin{eqnarray*}
&& \lim_{n \to \infty}\frac{1}{n^{1-2k(1-\alpha)}} \log \frac{\int_{B\cap[-a n^{1-\alpha}, a n^{1-\alpha}]}{e^{-nG_n^h(m+ s n^{\alpha-1})}ds}}{\int_{[-a n^{1-\alpha}, a n^{1-\alpha}]}{e^{-nG_n^h(m+ s n^{\alpha-1})}ds}} \\
&=& \lim_{n \to \infty}\frac{1}{n^{1-2k(1-\alpha)}} \log \frac{\int_{B\cap[-a n^{1-\alpha}, a n^{1-\alpha}]}{e^{-n(G_n^h(m+ s n^{\alpha-1})-G_n^h(m))}ds}}{\int_{[-a n^{1-\alpha}, a n^{1-\alpha}]}{e^{-n(G_n^h(m+ s n^{\alpha-1})-G_n^h(m))}ds}}\\
&\stackrel{!}{=}& -\inf_{x \in B}J(x)
\end{eqnarray*}
for every $B \in \mathcal{B}(\R)$. This obviously follows from proving
\begin{equation}\label{Kernaussage}
\lim_{n \to \infty}\frac{1}{n^{1-2k(1-\alpha)}} \log \int_{B\cap[-a n^{1-\alpha}, a n^{1-\alpha}]}{e^{-n(G_n^h(m+ s n^{\alpha-1})-G_n^h(m))}ds}
~=~ -\inf_{x \in B}J(x)
\end{equation}
for every $B \in \mathcal{B}(\R)$ as the choice $B = \R$ yields
\begin{equation*}
\lim_{n \to \infty}\frac{1}{n^{1-2k(1-\alpha)}} \log \int_{[-a n^{1-\alpha}, a n^{1-\alpha}]}{e^{-n(G_n^h(m+ s n^{\alpha-1})-G_n^h(m))}ds}
~=~ -\inf_{x \in \R}J(x) ~=~ 0.
\end{equation*}
Fix $B \in \mathcal{B}(\R)$ and let $B$ be non-empty as \eqref{Kernaussage} is trivial otherwise. Since $B$ is non-empty and $\lim_{|s|\rightarrow \infty}P_{2k}(s) = \infty$ there exists a constant $K > 0$ such that
\begin{equation}\label{eq:zweite Abschaetzung}
\inf_{|s| > K}P_{2k}(s) ~>~ \inf_{s \in B}\left(\frac{\lambda s^{2k}}{(2k)!}\right)+1.
\end{equation}
In particular, there exists a $x_0 \in B$ such that $|x_0| \leq K$ as otherwise
\begin{eqnarray*}
\inf_{s \in B}\left(\frac{\lambda s^{2k}}{(2k)!}\right)+1
&\geq& \frac{\lambda K^{2k}}{(2k)!}+1\\
&\geq& \frac{\lambda K^{2k}}{2 (2k)!}\\
&\geq& \inf_{|s| > K}P_{2k}(s)
\end{eqnarray*}
contradicting \eqref{eq:zweite Abschaetzung}. Since we have
\begin{eqnarray*}
&& \liminf_{n \to \infty}\frac{1}{n^{1-2k(1-\alpha)}} \log \int_{B\cap[-a n^{1-\alpha}, a n^{1-\alpha}]}{e^{-n(G_n^h(m+ s n^{\alpha-1})-G_n^h(m))}ds}\\
&\geq& \liminf_{n \to \infty}\frac{1}{n^{1-2k(1-\alpha)}} \log \int_{\substack{s \in B:\\ |s| \leq K}}{e^{-n(G_n^h(m+ s n^{\alpha-1})-G_n^h(m))}ds}
\end{eqnarray*}
and (cf. Lemma 1.2.15 in \cite{Demb1998})
\begin{eqnarray*}
&& \limsup_{n \to \infty}\frac{1}{n^{1-2k(1-\alpha)}} \log \int_{B\cap[-a n^{1-\alpha}, a n^{1-\alpha}]}{e^{-n(G_n^h(m+ s n^{\alpha-1})-G_n^h(m))}ds}\\
&=& \max\bigg\{\limsup_{n \to \infty}\frac{1}{n^{1-2k(1-\alpha)}} \log \int_{\substack{s \in B:\\ |s| \leq K}}{e^{-n(G_n^h(m+ s n^{\alpha-1})-G_n^h(m))}ds},\\
&& \limsup_{n \to \infty}\frac{1}{n^{1-2k(1-\alpha)}} \log \int_{\substack{s \in B:\\ K < |s| \leq a n^{1-\alpha}}}{e^{-n(G_n^h(m+ s n^{\alpha-1})-G_n^h(m))}ds}\bigg\}
\end{eqnarray*}
\eqref{Kernaussage} follows from proving
\begin{equation}\label{eq:Behauptung1}
\lim_{n \to \infty}\frac{1}{n^{1-2k(1-\alpha)}} \log \int_{\substack{s \in B:\\ |s| \leq K}}{e^{-n(G_n^h(m+ s n^{\alpha-1})-G_n^h(m))}ds} ~=~ -\inf_{x \in B}J(x)
\end{equation}
and
\begin{equation} \label{eq:Behauptung2}
\limsup_{n \to \infty}\frac{1}{n^{1-2k(1-\alpha)}} \log \int_{\substack{s \in B:\\ K < |s| \leq a n^{1-\alpha}}}{e^{-n(G_n^h(m+ s n^{\alpha-1})-G_n^h(m))}ds} ~\leq~ -\inf_{x \in B}J(x) -1.
\end{equation}
To see \eqref{eq:Behauptung1} notice that by Lemma \ref{uniform convergence}
\begin{equation*}
n^{2k(1-\alpha)} \big(G_n^h(m+ s n^{\alpha-1})-G_n^h(m)\big) ~\rightarrow~ \frac{\lambda s^{2k}}{(2k)!}
\end{equation*}
uniformly for $s \in [-K,K]$ as $n \to \infty$ and therefore
\begin{eqnarray*}
&& \lim_{n \to \infty}\frac{1}{n^{1-2k(1-\alpha)}} \log \int_{\substack{s \in B:\\ |s| \leq K}}{e^{-n(G_n^h(m+ s n^{\alpha-1})-G_n^h(m))}ds}\\
&=& \lim_{n \to \infty}\frac{1}{n^{1-2k(1-\alpha)}} \log \int_{\substack{s \in B:\\ |s| \leq K}}{e^{-n^{1-2k(1-\alpha)}\frac{\lambda}{(2k)!} s^{2k}}ds}.
\end{eqnarray*}
The latter is a well-studied object of Laplace's Method (use e.\,g.\ a slight generalization of Lemma 2.6 in \cite{Reic2012}) and we get
\begin{eqnarray*}
\lim_{n \to \infty}\frac{1}{n^{1-2k(1-\alpha)}} \log \int_{\substack{s \in B:\\ |s| \leq K}}{e^{-n(G_n^h(m+ s n^{\alpha-1})-G_n^h(m))}ds}
&=& -\inf_{\substack{s\in B:\\|s| \leq K}}\left(\frac{\lambda s^{2k}}{(2k)!}\right)\\
&=& -\inf_{s\in B} J(s),
\end{eqnarray*}
where the last equality follows from the existence of $x_0$ and the monotonicity of $J$.

\noindent What is left to prove is \eqref{eq:Behauptung2}. Since \eqref{eq:erste Abschaetzung} holds for all $|s| \leq a n^{1-\alpha}$ we immediately get for $n$ sufficiently large
\begin{eqnarray*}
&& \int_{\substack{s \in B:\\ K < |s| \leq a n^{1-\alpha}}}{e^{-n(G_n^h(m+ s n^{\alpha-1})-G_n^h(m))}ds} \\
&\leq& \int_{\substack{s \in B:\\ K < |s| \leq a n^{1-\alpha}}}{e^{-n^{1-2k(1-\alpha)}P_{2k}(s)}ds}\\
&\leq& 2 (an^{1-\alpha}-K) e^{-n^{1-2k(1-\alpha)}\inf_{K < |s| < an^{1-\alpha}}P_{2k}(s)}\\
&\leq& 2 an^{1-\alpha} e^{-n^{1-2k(1-\alpha)}\inf_{|s|>K}P_{2k}(s)}
\end{eqnarray*}
and consequently
\begin{equation*}
\limsup_{n \to \infty}\frac{1}{n^{1-2k(1-\alpha)}} \log \int_{\substack{s \in B:\\ K < |s| \leq a n^{1-\alpha}}}{e^{-n(G_n^h(m+ s n^{\alpha-1})-G_n^h(m))}ds} ~\leq~ -\inf_{|s|>K}P_{2k}(s).
\end{equation*}
Now, \eqref{eq:Behauptung2} follows from the choice of $K$ in \eqref{eq:zweite Abschaetzung}.
\end{proof}

\noindent Finally, we can turn to the proof of Theorem \ref{unconditioned version}. After proving Theorem \ref{conditioned version} this is quite easy.

\begin{proof}[Proof of Theorem \ref{unconditioned version}]
Let $h$ be such that the assertions of Theorem \ref{conditioned version} and Lemma \ref{Lemma:deMatos} hold. Fix $1-1/\big(2(2k-1)\big)<\alpha<1$ and let $A$ be the constant appearing in Theorem \ref{conditioned version}. Using Lemma \ref{transformation} and Lemma \ref{transfer principle} (i) it suffices to prove
\begin{eqnarray*}
&& \lim_{n \to \infty}\frac{1}{n^{1-2k(1-\alpha)}} \log \frac{\int_{B}{e^{-nG_n^h(m+ s n^{\alpha-1})}ds}}{\int_{\R}{e^{-nG_n^h(m+ s n^{\alpha-1})}ds}} \\
&=& \lim_{n \to \infty}\frac{1}{n^{1-2k(1-\alpha)}} \log \frac{\int_{B}{e^{-n(G_n^h(m+ s n^{\alpha-1})-G_n^h(m))}ds}}{\int_{\R}{e^{-n(G_n^h(m+ s n^{\alpha-1})-G_n^h(m))}ds}}\\
&\stackrel{!}{=}& -\inf_{x \in B}J(x)
\end{eqnarray*}
for every $B \in \mathcal{B}(\R)$. Again, this follows from proving
\begin{equation}\label{Kernaussage2}
\lim_{n \to \infty}\frac{1}{n^{1-2k(1-\alpha)}} \log \int_{B}{e^{-n(G_n^h(m+ s n^{\alpha-1})-G_n^h(m))}ds} ~=~ -\inf_{x \in B}J(x)
\end{equation}
for every $B \in \mathcal{B}(\R)$. As we have seen in the previous proof (see \eqref{Kernaussage}) for fixed $0 < a < A$
\begin{equation*}
\lim_{n \to \infty}\frac{1}{n^{1-2k(1-\alpha)}} \log \int_{B\cap[-an^{1-\alpha}, an^{1-\alpha}]}{e^{-n(G_n^h(m+ s n^{\alpha-1})-G_n^h(m))}ds} ~=~ -\inf_{x \in B}J(x)
\end{equation*}
and thus \eqref{Kernaussage2} follows by means of Lemma 1.2.15 in \cite{Demb1998} from
\begin{equation}\label{eq:letzte}
\lim_{n \to \infty}\frac{1}{n^{1-2k(1-\alpha)}} \log \int_{\substack{s \in B:\\ |s| > an^{1-\alpha}}}{e^{-n(G_n^h(m+ s n^{\alpha-1})-G_n^h(m))}ds} ~=~ -\infty.
\end{equation}
To see \eqref{eq:letzte} note
\begin{eqnarray}
&& \lim_{n \to \infty}\frac{1}{n^{1-2k(1-\alpha)}} \log \int_{\substack{s \in B:\\ |s| > an^{1-\alpha}}}{e^{-n(G_n^h(m+ s n^{\alpha-1})-G_n^h(m))}ds}\nonumber\\
&\leq & \lim_{n \to \infty}\frac{1}{n^{1-2k(1-\alpha)}} \log \int_{|s| > an^{1-\alpha}}{e^{-n(G_n^h(m+ s n^{\alpha-1})-G_n^h(m))}ds}\nonumber\\
&=& \lim_{n \to \infty}\frac{1}{n^{1-2k(1-\alpha)}} \log e^{n(G_n^h(m)-G(m))}\int_{|s-m| > a}{e^{-n(G_n^h(s)-G(m))}ds}.\label{eq:3}
\end{eqnarray}
Since $h$ is such that the assertions of Lemma \ref{Lemma:deMatos} hold, there exists $\varepsilon > 0$ such that for $n$ sufficiently large
\begin{eqnarray*}
G_n^h(m)-G(m) &\leq& \varepsilon \text{ and}\\
\int_{|s-m| > a}{e^{-n(G_n^h(s)-G(m))}ds} &\leq& e^{-2 \varepsilon n},
\end{eqnarray*}
where we have used that $m$ is the unique global minimum of $G$ so that the set $\R \smallsetminus [m-a,m+a]$ does not contain a global minimum of $G$. Therefore, \eqref{eq:3} yields
\begin{eqnarray*}
\lim_{n \to \infty}\frac{1}{n^{1-2k(1-\alpha)}} \log \int_{\substack{s \in B:\\ |s| > an^{1-\alpha}}}{e^{-n(G_n^h(m+ s n^{\alpha-1})-G_n^h(m))}ds} &\leq& - \varepsilon \lim_{n \to \infty} n^{2k(1-\alpha)}\\ &=& -\infty,
\end{eqnarray*}
since $\alpha <1$.
\end{proof}

\bibliographystyle{plain}
{\footnotesize \bibliography{MDPcurie-weiss}}
\end{document}